\documentclass[reqno]{amsart}

\pdfoutput=1
\usepackage{amsrefs}
\usepackage{amssymb,enumerate}
\usepackage{graphicx,color}

\newtheorem{theorem}{Theorem}
\newtheorem{lemma}[theorem]{Lemma}

\newtheorem{corollary}[theorem]{Corollary}

\theoremstyle{definition}
\newtheorem{definition}[theorem]{Definition}

\theoremstyle{remark}
\newtheorem{remark}[theorem]{\bf Remark}

\numberwithin{equation}{section}
\numberwithin{theorem}{section}

\newcommand{\intav}[1]{\mathchoice {\mathop{\vrule width 6pt height 3 pt depth  -2.5pt
\kern -8pt \intop}\nolimits_{\kern -6pt#1}} {\mathop{\vrule width
5pt height 3  pt depth -2.6pt \kern -6pt \intop}\nolimits_{#1}}
{\mathop{\vrule width 5pt height 3 pt depth -2.6pt \kern -6pt
\intop}\nolimits_{#1}} {\mathop{\vrule width 5pt height 3 pt depth
-2.6pt \kern -6pt \intop}\nolimits_{#1}}}

\newcommand{\intavl}[1]{\mathchoice {\mathop{\vrule width 6pt height 3 pt depth  -2.5pt
\kern -8pt \intop}\limits_{\kern -6pt#1}} {\mathop{\vrule width 5pt
height 3  pt depth -2.6pt \kern -6pt \intop}\nolimits_{#1}}
{\mathop{\vrule width 5pt height 3 pt depth -2.6pt \kern -6pt
\intop}\nolimits_{#1}} {\mathop{\vrule width 5pt height 3 pt depth
-2.6pt \kern -6pt \intop}\nolimits_{#1}}}

\newcommand{\R}{\mathbb{R}}
\newcommand{\T}{\mathbb{T}}
\newcommand{\N}{\mathbb{N}}
\newcommand{\Z}{\mathbb{Z}}

\begin{document}

\title[Skew products in infinite measure]{Ergodic properties of skew products\\ in infinite measure}

\author[Patr\'icia Cirilo]{Patr\'icia Cirilo}
\address{Universidade Estadual Paulista, Rua Crist\'ov\~ao Colombo 2265, 15054-000, S\~ao Jos\'e do Rio Preto, Brasil.}
\email{prcirilo@ibilce.unesp.br}

\author[Yuri Lima]{Yuri Lima}
\address{Weizmann Institute of Science, Faculty of Mathematics and Computer Science, POB 26, 76100, Rehovot, Israel.}
\email{yuri.lima@weizmann.ac.il}

\author[Enrique Pujals]{Enrique Pujals}
\address{Instituto Nacional de Matem\'atica Pura e Aplicada, Estrada Dona Castorina 110, 22460-320, Rio de Janeiro, Brasil.}
\email{enrique@impa.br}


\date{\today}


\begin{abstract}
Let $(\Omega,\mu)$ be a shift of finite type with a Markov probability, and
$(Y,\nu)$ a non-atomic standard measure space. For each symbol $i$ of the
symbolic space, let $\Phi_i$ be a measure-preserving automorphism of
$(Y,\nu)$. We study skew products of the form
$(\omega,y)\mapsto(\sigma\omega,\Phi_{\omega_0}(y))$, where $\sigma$ is
the shift map on $(\Omega,\mu)$. We prove that, when the skew product is conservative,
it is ergodic if and only if the $\Phi_i$'s have no common non-trivial invariant set.

In the second part we study the skew product when $\Omega=\{0,1\}^\Z$, $\mu$ is a Bernoulli
measure, and $\Phi_0,\Phi_1$ are $\R$-extensions of a same uniquely ergodic
probability-preserving automorphism. We prove that, for a large class of roof functions,
the skew product is rationally ergodic with return sequence asymptotic to $\sqrt{n}$,
and its trajectories satisfy the central, functional central and local limit theorem.
\end{abstract}

\maketitle

\section{Introduction and statement of results}

Let $\T=\R/\Z$ and $\mathbb A=\T\times\R$. Of course, for any $\alpha\in\T$
the transformation $\Phi_0:(x,t)\in\mathbb A\mapsto(x+\alpha,t)$ is not ergodic
wrt the Lebesgue measure on $\mathbb A$.
Now let $\beta\in\T$, $\phi:\T\rightarrow\R$ a $L^1$-function with zero mean,
and $\Phi_1:(x,t)\in\mathbb A\mapsto(x+\beta,t+\phi(x))$.
$\Phi_1$ also preserves the Lebesgue measure on $\mathbb A$. There are clear
obstructions for its ergodicity, e.g. when the equation $\phi(x)=\psi(x+\beta)-\psi(\beta)$
has a solution $\psi$.

In this paper we study ergodic properties of random iterations of such
transformations. Because the invariant foliations of $\Phi_0$ and $\Phi_1$ are
different, it can happen that the {\it random dynamical system} is ergodic.
The theorem below gives, in terms of $\phi$, checkable conditions for ergodicity.
Given a probability space $(X,\nu)$, let $L^1_0(X,\nu)$ denote the set of $L^1$-integrable
functions $\phi:X\rightarrow\R$ with zero mean, and let the {\it essential image} of
$\phi\in L^1_0(X,\nu)$ be the set of $t\in\R$ for which $\phi^{-1}[t-\varepsilon,t+\varepsilon]$
has positive $\nu$-measure for any $\varepsilon>0$.

\begin{theorem}\label{main thm 1}
Let $\mu$ be a Bernoulli measure on $\{0,1\}^\Z$, let $T_0,T_1$ be probability-preserving
automorphisms of a non-atomic standard probability space $(X,\nu)$, with $T_0$ ergodic,
and let $\phi\in L^1_0(X,\nu)$. Then
$$
\begin{array}{rcrcl}
F&:&\{0,1\}^\Z\times X\times\R&\longrightarrow &\{0,1\}^\Z\times X\times\R\\
 & &(\omega,x,t)              &\longmapsto     &(\sigma\omega,T_{\omega_0}x,t+\omega_0\phi(x))
\end{array}
$$
is ergodic iff the closed subgroup generated by the essential image of $\phi$ is $\R$.
\end{theorem}

Above and henceforth, we endow skew products with the product measure.
Theorem \ref{main thm 1} is consequence of a more general statement. Let
$\Omega\subset\{0,\ldots,k-1\}^\Z$.

\begin{theorem}\label{main thm 3}
Let $(\Omega,\mu)$ be a shift of finite type with a Markov probability, and
$\Phi_0,\ldots,\Phi_{k-1}$ measure-preserving automorphisms of a non-atomic
standard measure space $(Y,\nu)$. Assume that
$F:(\omega,y)\in\Omega\times Y\mapsto(\sigma\omega,\Phi_{\omega_0}(y))$ is conservative.
Then $F$ is ergodic iff $\Phi_0,\ldots,\Phi_{k-1}$ have no common non-trivial invariant set.
\end{theorem}

\begin{corollary}\label{main corollary}
Let $(\Omega,\mu)$ be a shift of finite type with a Markov probability, and
$T_0,\ldots,T_{k-1}$ probability-preserving automorphisms of a non-atomic
standard probability space $(X,\nu)$ with no common non-trivial invariant sets.
Let $\phi_i\in L^1_0(X,\nu)$ and
$\Phi_i:(x,t)\in X\times\mathbb R\mapsto(T_i(x),t+\phi_i(x))$, $i=0,\ldots,k-1$.
Then the skew product $(\omega,x,t)\mapsto(\sigma\omega,\Phi_{\omega_0}(x,t))$ is ergodic iff
$\Phi_0,\ldots,\Phi_{k-1}$ have no common non-trivial invariant set.
\end{corollary}

Theorem \ref{main thm 3} is related to a result of Kakutani \cite{kakutani1951random}.
Let $(S,\rho)$ be a probability space, $(\Omega,\mu)=(S^{\mathbb N},\rho^{\mathbb N})$, and
$\sigma:(\Omega,\mu)\rightarrow(\Omega,\mu)$ be the shift map. Kakutani proved that if $(Y,\nu)$ is a probability space
and $\{\Phi_s\}_{s\in S}$ is a measurable family of probability-preserving automorphisms of $(Y,\nu)$,
then $F:(\omega,y)\mapsto(\sigma\omega,\Phi_{\omega_0}(y))$ is ergodic iff $\{\Phi_s\}_{s\in S}$
have no common non-trivial invariant set. Observe that, in this case, $F$ is automatically conservative.

The first version of Kakutani's theorem for infinite measures appeared in a paper of Wo{\'s}
\cite{wos1982random}, also for Bernoulli systems of the form $(S^{\mathbb N},\rho^{\mathbb N})$.
Thus Theorem \ref{main thm 3} does not follow either from Kakutani's neither from Wo{\'s}' results.
We would like to thank David Sauzin for pointing us reference \cite{wos1982random}.
Indeed, he has a strong application of such result for the context of standard maps
\cite{sauzin2006ergodicity}.

Some classical theorems in ergodic theory are not valid for infinite measures. E.g. Birkhoff's
averages converge to zero almost surely, provided the transformation is conservative and ergodic.
This leads the following question: what is a candidate for Birkhoff-type theorem?
One attempt was made by Aaronson, who introduced the notion of {\it rational ergodicity}
(see \S\ref{subsection infinite ergodic theory} for the definition).
Given a function $f$, denote its Birkhoff sums by $S_nf$. Rationally ergodic maps possess a sort
of C\`esaro-averaged version of convergence in measure: there is a sequence $\{a_n\}_{n\ge 1}$ such that,
for every $L^1$-function $f$ and every sequence $\{n_k\}_{k\ge 1}$ of positive integers, there exists a
subsequence $\{n_{k_l}\}_{l\ge 1}$ such that $S_{n_{k_l}}f/a_{n_{k_l}}$ converges to $\int f$ almost everywhere.
This latter property is called {\it weak homogeneity} and the sequence $\{a_n\}_{n\ge 1}$ is called a
{\it return sequence}.

Many authors investigated ergodic transformations of $\mathbb A$
\cite{conze1980ergodicite,conze1976ergodicite,krygin1974examples,schmidt1976lectures,schmidt1978cylinder},
but few established rational ergodicity. Aaronson and Keane \cite{aaronson1982visitors} considered ``deterministic''
random walks driven by irrational rotations of $\mathbb T$, and showed that the associated skew product on $\mathbb A$
is rationally ergodic. In \cite{cirilo2011law} we constructed, for almost every $\alpha\in\mathbb R$,
skew products of the form $(x,t)\in\mathbb A\mapsto(x+\alpha,t+\phi(x))$ that are rationally ergodic
along a subsequence of iterates. Here we consider a special case of Theorem \ref{main thm 3} and prove
that the associated skew product is rationally ergodic.

\begin{theorem}\label{main thm 2}
Let $\Omega=\{0,1\}^\Z$, $\mu$ a Bernoulli measure on $\Omega$, $T$ a uniquely ergodic probability-preserving
automorphism of a non-atomic standard probability space $(X,\nu)$, and $\phi:X\rightarrow\mathbb R$ a
non-zero continuous function with
\begin{align}\label{assumption on phi}
\dfrac{1}{\sqrt{n}}\sum_{i=0}^{n-1}\phi(T^ix)\rightarrow 0\ \ \text{ uniformly in }x.
\end{align}
Then $(\omega,x,t)\mapsto(\sigma\omega,Tx,t+\omega_0\phi(x))$ is
rationally ergodic with return sequence $\sqrt{n}$, and its trajectories satisfy
central, functional central and local limit theorem.
\end{theorem}

Assumption (\ref{assumption on phi}) is natural for obtaining limit theorems, because the speed
of growth of $\sum_{i=0}^{n-1}\phi(T^ix)$ has to be lower than in simple random walks of $\Z$.
It holds e.g. when $\phi$ is a coboundary for $T$.

Let $F$ denote the skew product $(\omega,x,t)\mapsto(\sigma\omega,Tx,t+\omega_0\phi(x))$.
The third coordinate of $F^n$ is
\begin{align}\label{third coordinate of F^n}
t+\sum_{i=0}^{n-1}\omega_i\phi(T^ix)=t+\dfrac{1}{2}\sum_{i=0}^{n-1}\phi(T^ix)(2\omega_i-1)+\dfrac{1}{2}\sum_{i=0}^{n-1}\phi(T^ix).
\end{align}
Let $X_1,X_2,\ldots$ be independent identically distributed random variables, each with law
$\mathbb P[X_n=1]=\mathbb P[X_n=-1]=\frac{1}{2}$. For each $x\in X$, let $\{S_n^x\}_{n\ge 1}$ be
the martingale
\begin{align*}
S_n^x=\phi(x)\cdot X_0+\phi(Tx)\cdot X_1+\cdots+\phi(T^{n-1}x)\cdot X_{n-1}.
\end{align*}
Then (\ref{third coordinate of F^n}) equals $t+\frac{1}{2}(S_n^x+\sum_{i=0}^{n-1}\phi(T^ix))$.
Because $\phi$ satisfies (\ref{assumption on phi}) and $\{S_n^x\}_{n\ge 1}$ is a martingale with
bounded increments, the sequences $\sum_{i=0}^{n-1}\omega_i\phi(T^ix),n\ge 1$, satisfy
both the central and functional central limit theorem. Rational ergodicity does not follow
from these theorems. For that we need a local limit theorem.

\begin{theorem}\label{main thm 4}
Under the conditions of Theorem \ref{main thm 2}, let $\{s_n^x\}_{n\ge 1}\subset\mathbb R$ with
\begin{align*}
\lim_{n\rightarrow\infty}s_n^x/\sqrt{n}=0\ \ \text{uniformly in }x.
\end{align*}
Given $t>0$, there are $K,n_0>0$ such that
\begin{align*}
K^{-1}\le \sqrt{n}\cdot \mathbb P[S_n^x\in [-t,t]-s_n^x]\le K,\ \ \ \ \forall\,n>n_0,\forall\,x\in X.
\end{align*}
\end{theorem}

Theorem \ref{main thm 4} is a uniform local limit theorem with moving targets, where the increments
are independent but not identically distributed. Its proof uses Fourier analysis. See e.g. \S 10.4
of \cite{breiman1968probability} for Fourier analytical proofs of limit theorems.

Now consider a special case of Theorem \ref{main thm 1}: let $T_0,T_1$ be irrational
rotations of $\mathbb T$. When $\phi$ has small variation, $\Phi_1$ is a
conservative perturbation of $\Phi_0$, a particular situation that naturally appears
in the phenomenon called {\it Arnold diffusion}. In \cite{sauzin2006examples}, the author
proposed that a small perturbation in the Gevrey category of a non-degenerate
integrable Hamiltonian system gives rise to a dynamics that can be reduced to a skew
product extension of integrable transformations of $\mathbb A$ over $\{0,1\}^\Z$,
and proved that the trajectories of the skew product satisfy the functional central
limit theorem.

Our results apply to a slight variation of the model proposed in
\cite{sauzin2006examples}, when the integrable transformations of $\mathbb A$ are
$\R$-extensions of rotations of $\mathbb T$, and we also obtain a uniform local
limit theorem with moving targets (Theorem \ref{main thm 4}), and that the skew
product is rationally ergodic (Theorem \ref{main thm 2}). We believe these results
can be extended to the case treated in \cite{sauzin2006examples}.

The paper is organized as follows. In \S\ref{section preliminaries} we establish
the necessary preliminaries. In \S\ref{section proof of thm 3} we prove Theorem
\ref{main thm 3} and Corollary \ref{main corollary}. Section \ref{section proof of thm 1}
encloses the first part of the paper, where we prove Theorem \ref{main thm 1}.
The second part consists of \S\S\ref{section proof of thm 4} and
\ref{section proof of thm 2}: in \S\ref{section proof of thm 4} we prove
Theorem \ref{main thm 4}, and in \S\ref{section proof of thm 2} we prove
Theorem \ref{main thm 2}.

\section{Notation and preliminaries}\label{section preliminaries}

\begin{definition}\label{def vinogradov}
Let $f,g:\N\rightarrow\R$. We write $f\lesssim g$ if there is $C>0$ such that
$$|f(n)|\le C\cdot |g(n)|\,,\ \ \forall\,n\in\N.$$
If $f\lesssim g$ and $g\lesssim f$, we write $f\sim g$.
\end{definition}

Given an irreducible stochastic matrix $P=(p_{ij})_{0\le i,j<k}$, $\Omega=\Omega(P)$
is the {\it shift of finite type} with transition matrix $P$:
\begin{align*}
\Omega=\left\{(\ldots,\omega_{-1},\omega_0,\omega_1,\ldots)\in\{0,\ldots,k-1\}^\Z:
p_{\omega_i,\omega_{i+1}}>0 \text{ for all }i\in\Z\right\}.
\end{align*}
$\omega=(\omega_n)_{n\in\Z}$ denotes an element of $\Omega$.
Let $\sigma:\Omega\rightarrow\Omega$ be the left shift, i.e. $(\sigma\omega)_n=\omega_{n+1}$.
Given $\overline\omega\in\Omega$, a {\it cylinder} containing $\overline\omega$ is a set of the form
\begin{align*}
[\omega_n=\overline\omega_n,\ldots,\omega_m=\overline\omega_m]=
\{\omega\in\Omega:\omega_n=\overline\omega_n,\ldots,\omega_m=\overline\omega_m\}.
\end{align*}
Given a probability vector $\pi=(\pi_0,\ldots,\pi_{k-1})$, let $\mu$ be the probability on
$\Omega$ defined as
\begin{align*}
\mu[\omega_n=\overline\omega_n,\ldots,\omega_m=\overline\omega_m]=
\pi_{\overline\omega_n}p_{\overline\omega_n,\overline\omega_{n+1}}\cdots p_{\overline\omega_{m-1},\overline\omega_m}.
\end{align*}
$\mu$ is called a {\it Markov probability}. Clearly, it is invariant under $\sigma$.

The proof of Theorem \ref{main thm 3} will use the hyperbolic structure of $\Omega$.
We now setup the tools that will be needed.

\subsection{s-sets and u-sets}\label{subsection stable and unstable sets}

Let $\overline\omega\in\Omega$.
A {\it s-set} is a set of the form
\begin{align*}
[w_i=\overline w_i,i\ge n]=\{\omega\in\Omega\,:\,\omega_i=\overline\omega_i\text{ for all }i\ge n\}
\end{align*}
and a {\it u-set} is a set of the form
\begin{align*}
[w_i=\overline w_i,i\le n]=\{\omega\in\Omega\,:\,\omega_i=\overline\omega_i\text{ for all }i\le n\}.
\end{align*}

A cylinder $[\omega_n=\overline\omega_n]$ can be seen either as a union of s-sets or of u-sets:
\begin{align*}
[\omega_n=\overline\omega_n]=\bigcup_{\tilde\omega\in[\omega_i=\overline\omega_i,i\le n]}[\omega_i=\tilde\omega_i,i\ge n]=
\bigcup_{\tilde\omega\in[\omega_i=\overline\omega_i,i\ge n]}[\omega_i=\tilde\omega_i,i\le n].
\end{align*}
Furthermore, for any $\tilde\omega\in[\omega_n=\overline\omega_n]$ the intersections
\begin{align*}
[\omega_i=\overline\omega_i,i\ge n]\cap[\omega_i=\tilde\omega_i,i\le n]\ \text{ and }\
[\omega_i=\tilde\omega_i,i\ge n]\cap[\omega_i=\overline\omega_i,i\le n]
\end{align*}
consist of single points $\langle\overline\omega,\tilde\omega\rangle$ and $\langle\tilde\omega,\overline\omega\rangle$.

\begin{figure}[hbt!]
\def\svgwidth{10cm}
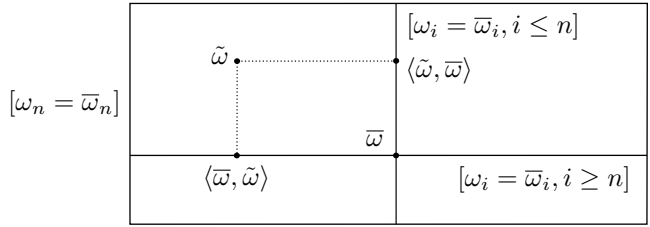
\caption{Local cylinder coordinates.}\label{picture 2}
\end{figure}

\noindent Thus the map
$$
\begin{array}{rcl}
[\omega_n=\overline\omega_n]&\longrightarrow &[\omega_i=\overline\omega_i,i\ge n]\times[\omega_i=\overline\omega_i,i\le n]\\
&&\\
\tilde\omega &\longmapsto&(\langle\overline\omega,\tilde\omega\rangle,\langle\tilde\omega,\overline\omega\rangle).
\end{array}
$$
is a bijection. Call it a {\it local cylinder coordinate} of $[\omega_n=\overline\omega_n]$.

Each s-set $[\omega_i=\overline\omega_i,i\ge n]$ is isomorphic to a one-sided symbolic space.
Its sigma-algebra is generated by the infinite cylinders of the form
\begin{align*}
[\omega_i=\tilde\omega_i,i\ge n-k],\ \text{ where }\tilde\omega\in[\omega_i=\overline\omega_i,i\ge n]\text{ and }k\ge 0.
\end{align*}
Call them {\it s-cylinders} of the s-set $[\omega_i=\overline\omega_i,i\ge n]$,
and $k$ the {\it length} of the s-cylinder. Of course, s-sets are s-cylinders of themselves,
and even more: a s-set is a s-cylinder of infinitely many s-sets.
Define {\it u-cylinders} in a similar way.

\subsection{s-measures and u-measures}\label{subsection s-measures}

Endow each s-set $[\omega_i=\overline\omega_i,i\ge n]$ with a {\it s-measure} $\mu^s$, defined
on its s-cylinders by
\begin{align*}
\mu^s[\omega_i=\tilde\omega_i,i\ge n-k]=\pi_{\tilde\omega_{n-k}}p_{\tilde\omega_{n-k},\tilde\omega_{n-k+1}}\cdots
p_{\tilde\omega_{n-1},\tilde\omega_n}.
\end{align*}
Similarly, define a {\it u-measure} $\mu^u$ on $[\omega_i=\overline\omega_i,i\le n]$ by
\begin{align*}
\mu^u[\omega_i=\tilde\omega_i,i\le n+k]=p_{\tilde\omega_n,\tilde\omega_{n+1}}p_{\tilde\omega_{n+1},\tilde\omega_{n+2}}
\cdots p_{\tilde\omega_{n+k-1},\tilde\omega_{n+k}}.
\end{align*}
$\mu^s$ and $\mu^u$ are one-sided Markov probabilities. A local cylinder coordinate
$[\omega_n=\overline\omega_n]\rightarrow[\omega_i=\overline\omega_i,i\ge n]\times[\omega_i=\overline\omega_i,i\le n]$
sends the restriction $\mu|_{[\omega_n=\overline\omega_n]}$ to the product measure $\mu^s\times\mu^u$.

One-sided Markov probabilities satisfy a {\it ratio preserving property}: if $A,B$ are subsets of
a cylinder of length $k$, then the quotient of the measures of their $k$-th iterates is preserved.
This is the content of the next lemma. Let $[\omega_i=\overline\omega_i,i\ge n-k]$ be a s-cylinder
of length $k$ of the s-set $[\omega_i=\overline\omega_i,i\ge n]$. Observe that
\begin{align*}
\sigma^{-k}[\omega_i=\overline\omega_i,i\ge n-k]=[\omega_i=\overline\omega_{i-k},i\ge n]
\end{align*}
is another s-set, and thus can be endowed with a s-measure $\mu^s$.

\begin{lemma}\label{lemma bounded distortion}
Let $[\omega_i=\overline\omega_i,i\ge n]$ be a s-set, and let $[\omega_i=\overline\omega_i,i\ge n-k]$
be a s-cylinder of length $k$. If $A,B\subset[\omega_i=\overline\omega_i,i\ge n-k]$, then
\begin{align}\label{equality bounded distortion}
\dfrac{\mu^s(\sigma^{-k}A)}{\mu^s(\sigma^{-k}B)}=\dfrac{\mu^s(A)}{\mu^s(B)}\,\cdot
\end{align}
Analogously, if $A,B$ are contained in a u-cylinder of length $k$ of a u-set, then
\begin{align*}
\dfrac{\mu^u(\sigma^kA)}{\mu^u(\sigma^kB)}=\dfrac{\mu^u(A)}{\mu^u(B)}\,\cdot
\end{align*}
\end{lemma}

\begin{proof}
The sigma-algebra on $[\omega_i=\overline\omega_i,i\ge n-k]$ is generated by s-cylinders of length $\ge k$.
Thus we can assume $A$ and $B$ are both s-cylinders of length $\ge k$. Take
$\tilde\omega,\hat\omega\in[\omega_i=\overline\omega_i,i\ge n-k]$, take $l,m\ge k$, and let
\begin{align*}
A=[\omega_i=\tilde\omega_i,i\ge n-l]\ \text{ and }\ B=[\omega_i=\hat\omega_i,i\ge n-m].
\end{align*}
We have
\begin{align*}
\sigma^{-k}A=[\omega_i=\tilde\omega_{i-k},i\ge n-l+k]\ \text{ and }
\ \sigma^{-k}B=[\omega_i=\hat\omega_{i-k},i\ge n-m+k].
\end{align*}
As s-cylinders of $[\omega_i=\overline\omega_{i-k},i\ge n]$, the quotient of their $\mu^s$-measures is
\begin{eqnarray*}
\dfrac{\mu^s(\sigma^{-k}A)}{\mu^s(\sigma^{-k}B)}&=&
\dfrac{\pi_{\tilde\omega_{n-l}}p_{\tilde\omega_{n-l},\tilde\omega_{n-l+1}}\cdots
p_{\tilde\omega_{n-k-1},\tilde\omega_{n-k}}}{\pi_{\hat\omega_{n-m}}p_{\hat\omega_{n-m},\hat\omega_{n-m+1}}\cdots
p_{\hat\omega_{n-k-1},\hat\omega_{n-k}}}\\
&&\\
&=&\dfrac{\pi_{\tilde\omega_{n-l}}p_{\tilde\omega_{n-l},\tilde\omega_{n-l+1}}\cdots
p_{\tilde\omega_{n-1},\tilde\omega_{n}}}{\pi_{\hat\omega_{n-m}}p_{\hat\omega_{n-m},\hat\omega_{n-m+1}}\cdots
p_{\hat\omega_{n-1},\hat\omega_{n}}}\\
&&\\
&=&\dfrac{\mu^s(A)}{\mu^s(B)}\,,
\end{eqnarray*}
where in the second equality we used that $\tilde\omega_i=\hat\omega_i=\overline\omega_i$ for $n-k\le i\le n$.
The other statement is proved similarly.
\end{proof}

The above lemma constitutes the first of three properties of s-sets, u-sets, s-measures and
u-measures we will need. The second is that non-trivial subsets of cylinders cannot be simultaneously
saturated by s-sets and u-sets.

\begin{lemma}\label{saturation lemma}
Let $A\subset[\omega_n=\overline\omega_n]$ with positive $\mu$-measure. If for $\mu$-almost every
$\tilde\omega\in A$ both
\begin{align*}
[\omega_i=\tilde\omega_i,i\ge n]\text{ and }[\omega_i=\tilde\omega_i,i\le n]\subset A,
\end{align*}
then $A=[\omega_n=\overline\omega_n]$.
\end{lemma}

\begin{proof}
Let $A'$ be the image of $A$ under the local cylinder coordinates
$[\omega_n=\overline\omega_n]\rightarrow[\omega_i=\overline\omega_i,i\ge n]\times[\omega_i=\overline\omega_i,i\le n]$.
Because $[\omega_i=\tilde\omega_i,i\ge n]\subset A$
for $\mu$-almost every $\tilde\omega\in A$, $A'$ is a product set of the form
$[\omega_i=\overline\omega_i,i\ge n]\times U$.
Because $[\omega_i=\tilde\omega_i,i\le n]\subset A$ for $\mu$-almost every $\tilde\omega\in A$,
$A'$ is also a product set of the form $S\times [\omega_i=\overline\omega_i,i\le n]$.
This clearly implies that
$A'=[\omega_i=\overline\omega_i,i\ge n]\times[\omega_i=\overline\omega_i,i\le n]$,
and then $A=[\omega_n=\overline\omega_n]$.
\end{proof}

The third property is a Lebesgue differentiation theorem.

\begin{lemma}\label{differentiation theorem}
Let $A\subset[\omega_n=\overline\omega_n]$. Then for $\mu$-almost every $\tilde\omega\in A$
\begin{align*}
\lim_{k\to\infty}\dfrac{\mu^s(A\cap[\omega_i=\tilde\omega_i,i\ge n-k])}{\mu^s[\omega_i=\tilde\omega_i,i\ge n-k]}=
\lim_{k\to\infty}\dfrac{\mu^u(A\cap[\omega_i=\tilde\omega_i,i\le n+k])}{\mu^u[\omega_i=\tilde\omega_i,i\le n+k]}=1.
\end{align*}

\end{lemma}

\begin{proof}
Fix a s-set $[\omega_i=\overline\omega_i,i\ge n]$, and let
\begin{align*}
\mathcal P_k=\{[\omega_i=\tilde\omega_i,i\ge n-k]:\tilde\omega\in[\omega_i=\overline\omega_i,i\ge n]\}
\end{align*}
be its partition into s-cylinder of length $k$. $\bigvee_{k\ge 0}\mathcal P_k$ equals
the sigma-algebra on $[\omega_i=\overline\omega_i,i\ge n]$.
For each $k\ge 0$, let $\mathcal F_k$ be the sigma-algebra generated by $\mathcal P_k$.
$\{\mathcal F_k\}_{k\ge 0}$ is a filtration on $[\omega_i=\overline\omega_i,i\ge n]$.

For $\tilde\omega\in[\omega_i=\overline\omega_i,i\ge n]$,
$\alpha_k(\tilde\omega)=[\omega_i=\tilde\omega_i,i\ge n-k]$ is
the element of $\mathcal P_k$ containing $\tilde\omega$. For any measurable bounded function
$f:[\omega_i=\overline\omega_i,i\ge n]\rightarrow\mathbb R$, the sequence of functions
$\{\mathbb E[f|\mathcal F_k]\}_{k\ge 0}$ converges pointwise $\mu^s$-almost surely to $f$,
by the martingale convergence theorem. When $f=\chi_A$,
\begin{align*}
\mathbb E[f|\mathcal F_k](\tilde\omega)=\dfrac{1}{\mu^s(\alpha_k(\tilde\omega))}\int_{\alpha_k(\tilde\omega)}fd\mu^s=
\dfrac{\mu^s(A\cap[\omega_i=\tilde\omega_i,i\ge n-k])}{\mu^s[\omega_i=\tilde\omega_i,i\ge n-k]}
\end{align*}
and so
\begin{align}\label{lebesgue diff for s-measure}
\lim_{k\to\infty}\dfrac{\mu^s(A\cap[\omega_i=\tilde\omega_i,i\ge n-k])}{\mu^s[\omega_i=\tilde\omega_i,i\ge n-k]}=\chi_A(\tilde\omega)
\end{align}
for $\mu^s$-almost every $\tilde\omega\in[\omega_i=\overline\omega_i,i\ge n]$.

By a similar argument,
\begin{align}\label{lebesgue diff for u-measure}
\lim_{k\to\infty}\dfrac{\mu^u(A\cap[\omega_i=\tilde\omega_i,i\le n+k])}{\mu^u[\omega_i=\tilde\omega_i,i\le n+k]}=\chi_A(\tilde\omega)
\end{align}
for $\mu^u$-almost every $\tilde\omega\in[\omega_i=\overline\omega_i,i\le n]$.
Because the local cylinder coordinates send $\mu|_{[\omega_n=\overline\omega_n]}$ to $\mu^s\times\mu^u$,
relations (\ref{lebesgue diff for s-measure}) and (\ref{lebesgue diff for u-measure}) give the result.
\end{proof}

\subsection{Infinite ergodic theory}\label{subsection infinite ergodic theory}

Let $\Phi$ be an ergodic measure-preserving automorphism of a non-atomic standard
measure space $(Y,\nu)$. Assume that $\Phi$ is {\it conservative}: $\nu(A)=0$ for
any measurable $A\subset Y$ such that $\{\Phi^{-n}A\}_{n\ge 0}$ are pairwise disjoint.

As stated in the introduction, for every $f\in L^1(Y,\nu)$ the Birkhoff
averages $S_nf(y)/n$ converge to zero $\nu$-almost everywhere.
Nevertheless, Hopf's ratio ergodic theorem is an indication that some sort of
regularity might exist and it might still be possible, for a specific sequence $\{a_n\}_{n\ge 1}$,
to smooth out the fluctuations of $S_nf/a_n$ by means of a summability method.

One attempt to obtain this was made by Aaronson, who introduced the
notion of rational ergodicity (see \S 3.3 of \cite{aaronson1997introduction}).
Given a measurable set $A\subset Y$, let $R_n:A\rightarrow\N$ be the return
function of $A$ with respect to $\Phi$:
\begin{align*}
R_n(y)=\#\{1\le i\le n:\Phi^i(y)\in A\}.
\end{align*}

\begin{definition}\label{def rationally ergodic}
A conservative ergodic measure-preserving automorphism $\Phi$ of a non-atomic standard
measure space $(Y,\nu)$ is called {\it rationally ergodic} if there is a measurable set
$A\subset Y$ with $0<\nu(A)<\infty$ such that the return function $R_n:A\rightarrow\N$
satisfies a {\it Renyi inequality}:
\begin{align*}
\int_A R_n^2d\nu\ \lesssim\ \left(\int_A R_nd\nu\right)^2.
\end{align*}
\end{definition}

Aaronson \cite{aaronson1977ergodic} (see also Theorem 3.3.1 of \cite{aaronson1997introduction})
proved that every rationally ergodic automorphism is {\it weakly homogeneous}: if
$\{a_n\}_{n\ge 1}$ is defined by
\begin{align}\label{return sequence}
a_{n}=\dfrac{1}{\nu(A)^2}\int_A R_nd\nu=\dfrac{1}{\nu(A)^2}\sum_{i=1}^{n}\nu\left(A\cap\Phi^{-i}A\right),
\end{align}
then every sequence $\{n_k\}_{k\ge 1}$ of positive integers can be refined to a
subsequence $\{n_{k_l}\}_{l\ge 0}$ such that for all $f\in L^1(Y,\nu)$ it holds
\begin{align*}
\dfrac{1}{N}\sum_{l=1}^N\dfrac{1}{a_{n_{k_l}}}S_{n_{k_l}}f(y)\longrightarrow\ \int_Y fd\nu\ \ \text{a.e.}
\end{align*}
$\{a_n\}_{n\ge 1}$ is called a {\it return sequence} of $\Phi$ and it is unique up to asymptotic equality.

We conclude these preliminaries stating a result that will be used in the next section.

\begin{theorem}[Atkinson \cite{atkinson1976recurrence}]\label{theorem atkinson}
Let $T$ be an ergodic probability-preserving automorphism of a non-atomic
standard probability space $(X,\nu)$, and let $\phi\in L^1_0(X,\nu)$.
Then $\nu$-almost every $x\in X$ has the following property: for any measurable
set $A\subset X$ containing $x$ with $\nu(A)>0$ and any $\varepsilon>0$, the set
$$\left\{n\ge 1:T^nx\in A\text{ and }|S_n\phi(x)|<\varepsilon\right\}$$
is infinite.
\end{theorem}

In other words, the $\R$-extension $(x,t)\mapsto(Tx,t+\phi(x))$ is conservative.

\section{Kakutani's theorem: proof of Theorem \ref{main thm 3}}\label{section proof of thm 3}

Call $\{\Phi_0,\ldots,\Phi_{k-1}\}$ an {\it ergodic system} if $\Phi_0,\ldots,\Phi_{k-1}$ have no
common non-trivial invariant set: every measurable set $A\subset Y$ such that
\begin{align*}
A=\Phi_0^{-1}A=\Phi_1^{-1}A=\cdots=\Phi_{k-1}^{-1}A
\end{align*}
either has zero or full $\nu$-measure. Alternatively, any $g\in L^\infty(Y,\nu)$ such that
$g\circ \Phi_0=\cdots=g\circ \Phi_{k-1}=g$ is constant almost everywhere.

Here we assume the skew product
$$
\begin{array}{rcrcl}
F&:&\Omega\times Y&\longrightarrow &\Omega\times Y\\
 & &(\omega,y)        &\longmapsto     &(\sigma\omega,\Phi_{\omega_0}(y))
\end{array}
$$
is conservative and we want to prove that $F$ is ergodic if and only if
$\{\Phi_0,\ldots,\Phi_{k-1}\}$ is an ergodic system. Clearly, if $F$ is ergodic then also is
$\{\Phi_0,\ldots,\Phi_{k-1}\}$. For instance, if $g(y)$ is invariant simultaneously for
$\Phi_0,\ldots,\Phi_{k-1}$, then $f(\omega,y)=g(y)$ is $F$-invariant.

We claim the converse is equivalent to prove that any bounded $F$-invariant function $f(\omega,y)$
does not depend on the first coordinate, i.e. there is a bounded function $g(y)$ such that
\begin{align}\label{independence on first coordinate}
f(\omega,y)=g(y)\ \ \text{\rm a.e.}
\end{align}
Indeed, if we assume this and let $f(\omega,x,t)=g(x,t)$ be $F$-invariant, then
whenever $\omega_0=i$ we get
\begin{align*}
(g\circ\Phi_i)(y)=f(\sigma\omega,\Phi_i(y))=(f\circ F)(\omega,y)=f(\omega,y)=g(y)
\end{align*}
and so $g$ is $\Phi_i$-invariant. By assumption, $g$ is constant almost everywhere
and thus also is $f$.

Fix a set $A\subset \Omega\times Y$ of positive measure, invariant
under $F$. In terms of characteristic functions, condition
(\ref{independence on first coordinate}) translates to saying that $A=\Omega\times B$ for
some $B\subset Y$. Alternatively, we define $A_y\subset\Omega$ by
\begin{align*}
A=\bigcup_{y\in Y}A_y\times\{y\}
\end{align*}
and want to prove that $A_y=\Omega$ for almost every $(\omega,y)\in A$.
We prove this using the tools developed in \S\ref{section preliminaries}.

\begin{lemma}\label{lemma saturation of cylinders}
If $\mu(A_y\cap[\omega_n=\overline\omega_n])>0$, then $[\omega_n=\overline\omega_n]\subset A_y$.
\end{lemma}

Assume Lemma \ref{lemma saturation of cylinders} has been proved. Each non-trivial $A_y$ intersects
some cylinder $[\omega_0=\overline\omega_0]$, and then $[\omega_0=\overline\omega_0]\subset A_y$.
Because $\Omega=\Omega(P)$ and $P$ is an irreducible matrix, there is $n\ge 1$ such that
\begin{align*}
\mu([\omega_0=\overline\omega_0]\cap[\omega_n=\tilde\omega_n])>0\ \ \text{ for any }\overline\omega,\tilde\omega\in\Omega.
\end{align*}
In particular, $\nu(A_y\cap[\omega_n=\tilde\omega_n])>0$ for any $\tilde\omega\in\Omega$.
Again by Lemma \ref{lemma saturation of cylinders}, it follows that
$[\omega_n=\tilde\omega_n]\subset A_y$ for any $\tilde\omega\in\Omega$, and so
\begin{align*}
\Omega=\bigcup_{\tilde\omega\in\Omega}[\omega_n=\tilde\omega_n]\subset A_y,
\end{align*}
thus proving that $A_y=\Omega$.

\begin{proof}[Proof of Lemma \ref{lemma saturation of cylinders}.]
According to Lemma \ref{saturation lemma}, it is enough to prove that
\begin{align}\label{inclusion stable and unstable}
[\omega_i=\hat\omega_i,i\ge n]\ \text{ and }\ [\omega_i=\hat\omega_i,i\le n]\subset A_y
\end{align}
for almost every $\hat\omega\in A_y$. Define measurable functions $\{f_k\}_{k\ge 1}$ on $A$ by
\begin{align*}
f_k(\tilde\omega,\tilde y)=\dfrac{\mu^u\left(A_{\tilde y}\cap[\omega_i=\tilde\omega_i,i\le n+k]\right)}{\mu^u[\omega_i=\tilde\omega_i,i\le n+k]}\,\cdot
\end{align*}
By Lemma \ref{differentiation theorem},
\begin{align}\label{limit densities}
\lim_{k\rightarrow\infty}f_k(\tilde\omega,\tilde y)=1\ \ \ \ \ \text{a.e. }(\tilde\omega,\tilde y)\in A.
\end{align}

Assume first that (\ref{limit densities}) holds uniformly in $A$.
Fix $\delta>0$ and let $k_0\ge 1$ for which $f_k>1-\delta$ for all $k>k_0$.
Because $F$ is conservative, for almost every $(\hat\omega,y)\in A$ there is $k>k_0$ such that
$(\tilde\omega,\tilde y)=F^{-k}(\hat\omega,y)\in A$, and then
\begin{align}\label{density in preimage}
\dfrac{\mu^u\left(A_{\tilde y}\cap[\omega_i=\tilde\omega_i,i\le n+k]\right)}{\mu^u[\omega_i=\tilde\omega_i,i\le n+k]}>1-\delta.
\end{align}

\begin{figure}[hbt!]
\def\svgwidth{12.4cm}
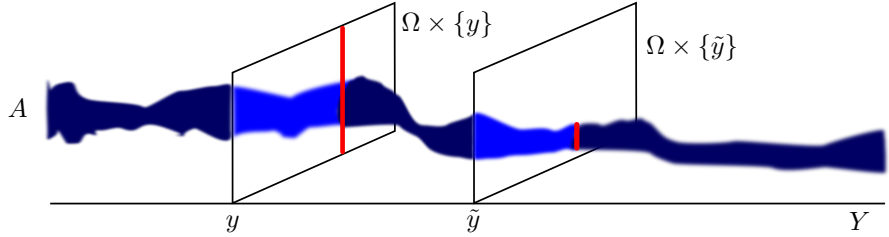
\caption{The saturation of $\Omega\times\{y\}$.}\label{picture 1}
\end{figure}

Because $F^k([\omega_i=\tilde\omega_i,i\le n+k]\times\{\tilde y\})=[\omega_i=\hat\omega_i,i\le n]\times\{y\}$,
Lemma \ref{lemma bounded distortion} and relation (\ref{density in preimage}) give that
\begin{eqnarray*}
\mu^u\left(A_y\cap [\omega_i=\hat\omega_i,i\le n]\right)
&=&\dfrac{\mu^u\left(A_y\cap [\omega_i=\hat\omega_i,i\le n]\right)}{\mu^u[\omega_i=\hat\omega_i,i\le n]}\\
&&\\
&=&\dfrac{\mu^u\left(\sigma^k\left(A_{\tilde y}\cap[\omega_i=\tilde\omega_i,i\le n+k]\right)\right)}
{\mu^u(\sigma^k[\omega_i=\tilde\omega_i,i\le n+k])}\\
&&\\
&=&\dfrac{\mu^u\left(A_{\tilde y}\cap[\omega_i=\tilde\omega_i,i\le n+k]\right)}{\mu^u[\omega_i=\tilde\omega_i,i\le n+k]}\\
&&\\
&>&1-\delta.
\end{eqnarray*}
Both $(\hat\omega,y)\in A$ and $\delta>0$ are arbitrary, and thus
$[\omega_i=\hat\omega_i,i\le n]\subset A_y$ for almost every $(\hat\omega,y)\in A$.
Analosgously, $[\omega_i=\hat\omega_i,i\ge n]\subset A_y$ for almost every $(\hat\omega,y)\in A$,
and this establishes (\ref{inclusion stable and unstable}).

In general, the convergence in (\ref{limit densities}) is not uniform. Instead, do the following:
for each $A'\subset A$ with finite measure and each $\varepsilon>0$, Egorov's theorem assures
the existence of $A''\subset A'$ such that
\begin{enumerate}
\item $(\mu\times\nu)(A'\backslash A'')<\varepsilon$, and
\item $\{f_k\}_{k\ge 1}$ converges uniformly in $A''$.
\end{enumerate}
By the previous argument, (\ref{inclusion stable and unstable}) holds almost everywhere in $A''$.
This concludes the proof of the lemma.
\end{proof}

\begin{remark}
In \cite{wos1982random}, Wo{\' s} proved a random ergodic theorem for sub-Markovian operators
in $L^\infty$. Because Koopman-von Neumann operators of measure-preserving automorphisms of
non-atomic standard probability spaces are always sub-Markov\-ian, his result characterizes
ergodicity for random dynamical systems over Bernoulli systems. It should be
interesting to mix our tools with Wo{\'s}' in order to extend his theorem to
skew products over shifts of finite type.
\end{remark}

It is not clear to us under which conditions $F$ is conservative. For instance, it can happen that
each $\Phi_i$ is conservative and $F$ is not. Here is an example communicated by Jon Aaronson:
let $Y=\{-1,1\}^\Z\times\Z^3$ and $\nu=$ Bernoulli measure on $\{-1,1\}^\Z\times$ counting measure
on $\Z^3$, and let $\Phi_0,\Phi_1,\Phi_2$ be measure-preserving transformations on $(Y,\nu)$ given by
\begin{eqnarray*}
\Phi_0(\theta,y)&=&(\varrho(\theta),y+\theta_0(1,0,0))\\
\Phi_1(\theta,y)&=&(\varrho(\theta),y+\theta_0(0,1,0))\\
\Phi_2(\theta,y)&=&(\varrho(\theta),y+\theta_0(0,0,1)),
\end{eqnarray*}
where $\varrho$ is the shift map on $\{-1,1\}^\Z$ and $\theta\in\{-1,1\}^\Z$.
Each $\Phi_i$ is isomorphic to a random walk on $\Z$, and so is conservative.
But $F$ is a random walk on $\Z^3$, which is not conservative.

Corollary \ref{main corollary} considers a class of conservative transformations for which
the skew product is conservative, as we'll now see.

\begin{proof}[Proof of Corollary \ref{main corollary}.]
By Theorem \ref{main thm 3}, we just need to prove that $F$ is conservative. Consider the skew product
$$
\begin{array}{rcrcl}
H&:&\Omega\times X&\longrightarrow &\Omega\times X\\
 & &(\omega,x)    &\longmapsto     &(\sigma\omega,T_{\omega_0}x).
\end{array}
$$
$H$ is a measure-preserving transformation in the probability space
$(\Omega\times X,\mu\times\nu)$. In particular, it is conservative. By assumption,
$\{T_0,\ldots,T_{k-1}\}$ is an ergodic system. Thus, Theorem \ref{main thm 3} implies
that $H$ is ergodic.

Now note that $F(\omega,x,t)=(H(\omega,x),t+\phi_{\omega_0}(x))$
is a skew product over $H$ and
\begin{align}\label{condition joint conservative}
\int_{\Omega\times X}\phi_{\omega_0}(x)d\mu(\omega)d\nu(x)=\sum_{i=0}^{k-1}\mu([\omega_0=i])\int_X\phi_i(x)d\nu(x)=0.
\end{align}
By Theorem \ref{theorem atkinson}, it follows that $F$ is conservative, and the proof is finished.
\end{proof}

Corollary \ref{main corollary} holds whenever the $\phi_i$'s satisfy equality
(\ref{condition joint conservative}). This is also a necessary condition. For example,
let $\phi_0=0$ and $\phi_1$ without zero mean such that the closed subgroup generated
by the essential image of $\phi_1$ is $\R$. By Theorem \ref{main thm 1}
(to be proved in \S\ref{section proof of thm 1}), $\{\Phi_0,\Phi_1\}$ is an
ergodic system. By Theorem \ref{theorem atkinson}, $F$ is not conservative. If $F$ is
also ergodic, then it is isomorphic to the translation $n\mapsto n+1$ on the
integers (see Proposition 1.2.1 of \cite{aaronson1997introduction}).
But we can choose $\phi_1$ properly such that this is not the case.

\section{Proof of Theorem \ref{main thm 1}}\label{section proof of thm 1}

Let $G$ be the closed subgroup generated by the essential image of $\phi$.
$G$ is either equal to $\alpha\Z$ or $\R$. Assume $G=\alpha\Z$.
If $\alpha=0$, then $F$ is clearly not ergodic.
If $\alpha\not=0$, then
\begin{align*}
A=\Omega\times X\times\left(\alpha\Z+[0,\alpha/4]\right)
\end{align*}
is a non-trivial $F$-invariant set, and again $F$ is not ergodic.

Now assume $G=\R$. We want to prove that $F$ is ergodic. Let
\begin{align*}
\Phi_0(x,t)=(T_0x,t)\ \ \text{ and }\ \ \Phi_1(x,t)=(T_1x,t+\phi(x)).
\end{align*}
By Corollary \ref{main corollary}, it is enough to prove that $\{\Phi_0,\Phi_1\}$ is an ergodic system.
Let $g(x,t)$ be a bounded function, invariant under $\Phi_0$ and $\Phi_1$. Then
\begin{align*}
g(T_0x,t)=(g\circ\Phi_0)(x,t)=g(x,t).
\end{align*}
Because $T_0$ is ergodic, $g$ does not depend on the first coordinate,
i.e. there is $h(t)$ such that $g(x,t)=h(t)$ almost everywhere.
It remains to prove that $h$ is constant almost everywhere. Note that
\begin{align*}
h(t+\phi(x))=g(T_1x,t+\phi(x))=(g\circ\Phi_1)(x,t)=g(x,t)=h(t)
\end{align*}
and so $h(t+\phi(x))=h(t)$ for almost every $t\in\R$ and almost every $x\in X$.
Thus the set
\begin{align*}
\mathcal P=\{s\in\R:h(t+s)=h(t)\text{ for almost every }t\in\R\}
\end{align*}
contains the essential image of $\phi$.

We claim that $\mathcal P$ is a closed subgroup of $\R$. It is clearly a subgroup.
By the Riesz representation theorem,
\begin{align*}
\mathcal P=\left\{s\in\R:\int_{\R}h(t+s)u(t)dt=\int_{\R}h(t)u(t)dt\text{ for every }u\in C_c(\R)\right\},
\end{align*}
where $C_c(\R)$ is the set of continuous functions $u:\R\rightarrow\R$ of compact support.
By the dominated convergence theorem, $\mathcal P$ is closed. Thus $\mathcal P=\R$,
i.e. $h$ is constant almost everywhere. This concludes the proof.

\section{Local limit theorem: proof of Theorem \ref{main thm 4}}\label{section proof of thm 4}

We now prove Theorem \ref{main thm 4}. To simplify notation, denote $c_i^x=\phi(T^ix)$ and
\begin{align*}
s_n^x=\sum_{i=0}^{n-1}\phi(T^ix)=\sum_{i=0}^{n-1}c_i^x.
\end{align*}
As we have seen in the introduction, the third coordinate of $F^n(\omega,x,t)$ is equal to
\begin{align*}
t+\sum_{i=0}^{n-1}\omega_i c_i^x=t+\dfrac{1}{2}\sum_{i=0}^{n-1}c_i^x X_i+\dfrac{s_n^x}{2}=t+\frac{1}{2}S_n^x+\frac{1}{2}s_n^x\,,
\end{align*}
where $\{S_n^x\}_{n\ge 1}$ is the martingale defined by
\begin{align*}
S_n^x=c_0^x X_0+\cdots+c_{n-1}^x X_{n-1}
\end{align*}
and $\{X_n\}_{n\ge 1}$ are independent identically distributed random variables, each with law
$\mathbb P[X_n=1]=\mathbb P[X_n=-1]=\frac{1}{2}$.

Because $\phi$ satisfies (\ref{assumption on phi}) and $\{S_n^x\}_{n\ge 1}$ is a martingale with
bounded increments, the sequences $\sum_{i=0}^{n-1}\omega_i\phi(T^ix),n\ge 1$, satisfy
both the central and functional central limit theorem \cite{hall1980martingale}.
Thus the trajectories of $F$ have a normal diffusion. Furthermore, because the trajectories
of $T$ equidistribute in $X$, $\{S_n^x\}_{n\ge 1}$ satisfies the local limit theorem
\begin{align*}
\lim_{n\to\infty}\sqrt{2\pi n}\cdot\mathbb P[S_n^x\in [a,b]]=b-a
\end{align*}
and even the local limit theorem with moving targets
\begin{align}\label{weak LLT}
\lim_{n\to\infty}\sqrt{2\pi n}\cdot\mathbb P[S_n^x\in [a,b]-s_n]=b-a
\end{align}
where $\{s_n\}_{n\ge 1}$ is a sequence such that $s_n/\sqrt{n}\rightarrow 0$.
The proof is similar to those in \S 10.4 of \cite{breiman1968probability}.

The local limit theorems above do not imply rational ergodicity, because different $x$'s
may give different rates of convergence. Rational ergodicity does not take into
account multiplicative constants, so what we need is to bound the expression in
the limit (\ref{weak LLT}) away from zero and infinity, uniformly in both $x$ and $n$.
This is the content of Theorem \ref{main thm 4}, which we'll now prove.

We assume, after a proper dilation, that $t=1$.
The proof proceeds as follows: firstly, we use the unique ergodicity of $T$ to estimate
the characteristic function of $S_n^x$, uniformly in $x$ and $n$.
Secondly, we use Fourier analysis and this estimate to establish the result.

Given a random variable $Y$, let $\varphi_Y:\R\rightarrow\mathbb C$ be its characteristics function:
\begin{align*}
\varphi_Y(t)=\mathbb E[\exp(itY)].
\end{align*}

\begin{lemma}\label{lemma estimate characteristic}
There exist $\delta,a,b,n_0>0$ such that for every $n>n_0$ and every $x\in X$
\begin{align*}
\exp(-at^2)\le \varphi_{S_n^x}\left(\dfrac{t}{\sqrt{n}}\right)\le\exp(-bt^2)\,,\ \ \forall\,|t|\le\delta\sqrt{n}.
\end{align*}
\end{lemma}

\begin{proof}
We have
\begin{align*}
\varphi_{X_0}(t)=\cos t=1-\dfrac{t^2}{2}+O(t^4)
\end{align*}
and so, for $|t|$ small,
\begin{align*}
\log\varphi_{X_0}(t)\le \log\left(1-\dfrac{t^2}{4}\right)\le -\dfrac{t^2}{8}
\end{align*}
and
\begin{align*}
\log\varphi_{X_0}(t)\ge \log(1-t^2)\ge -2t^2.
\end{align*}
Let $C=\sup_{x\in X}|\phi(x)|$ and take $\delta>0$ small so that
\begin{align}\label{estimate characteristic function}
\exp(-2t^2)\le\varphi_{X_0}(t)\le\exp(-t^2/8)\,,\ \ \forall\,|t|<\delta C.
\end{align}
Because
\begin{align*}
\varphi_{S_n^x}\left(\dfrac{t}{\sqrt{n}}\right)=\varphi_{X_0}\left(\dfrac{c_0^xt}{\sqrt{n}}\right)\cdots\varphi_{X_0}\left(\dfrac{c_{n-1}^xt}{\sqrt{n}}\right),
\end{align*}
(\ref{estimate characteristic function}) implies that, for every $|t|<\delta\sqrt{n}$,
\begin{align*}
\exp\left(-\frac{2\sum_{i=0}^{n-1}(c_i^x)^2}{n}\cdot t^2\right)\le
\varphi_{S_n^x}\left(\dfrac{t}{\sqrt{n}}\right)\le\exp\left(-\frac{\sum_{i=0}^{n-1}(c_i^x)^2}{8n}\cdot t^2\right)\cdot
\end{align*}
By Birkhoff's ergodic theorem, there is $n_0>0$ such that
\begin{align*}
\dfrac{1}{2}\int\phi^2d\nu\le \dfrac{1}{n}\displaystyle\sum_{i=0}^{n-1}(c_i^x)^2\le 2\int\phi^2d\nu\,,\ \ \
\forall\,n\ge n_0,\forall\,x\in X.
\end{align*}
Take
\begin{align*}
a=4\int\phi^2d\nu\ \text{ and }\ b=\dfrac{1}{16}\int\phi^2d\nu
\end{align*}
to conclude the proof of the lemma.
\end{proof}

Let $\chi_{[-1,1]}$ denote the indicator function of the interval $[-1,1]$, and
fix functions $g,h:\R\rightarrow\R$ such that\footnote{$\hat g,\hat h$ denote the Fourier transforms of $g,h$.}
\begin{enumerate}[(i)]
\item $g\le \chi_{[-1,1]}\le h$,
\item $\hat{g}(0)>0$, and
\item $\hat{g},\hat{h}$ are continuous with support contained in $[-\Delta,\Delta]$, for some
$\Delta>0$.
\end{enumerate}
It is not hard to exhibit such functions. Take, for example,
\begin{align*}
g=\dfrac{1}{12}\left[\left(\dfrac{\widehat{\chi_{[-4,4]}}}{4}\right)^4-
\left(\dfrac{\widehat{\chi_{[-4,4]}}}{4}\right)^2\right]\ \text{ and }\
h=\widehat{\chi_{[-1,1]}}^2.
\end{align*}
By (ii) and (iii), we can assume that $\delta>0$ satisfies
\begin{enumerate}[(iv)]
\item $\hat g|_{[-\delta,\delta]}>\hat{g}(0)/2$.
\end{enumerate}

\begin{proof}[Proof of Theorem \ref{main thm 4}]
We want to estimate
\begin{align*}
\sqrt{n}\cdot\mathbb P\left[S_n^x\in [-1,1]-s_n^x\right]=
\sqrt{n}\cdot\mathbb E\left[\chi_{[-1,1]}(S_n^x+s_n^x)\right].
\end{align*}
Because
\begin{align*}
\sqrt{n}\cdot\mathbb E[g(S_n^x+s_n^x)]\le\sqrt{n}\cdot\mathbb E\left[\chi_{[-1,1]}(S_n^x+s_n^x)\right]
\le\sqrt{n}\cdot\mathbb E[h(S_n^x+s_n^x)]\,,
\end{align*}
it is enough to estimate $\sqrt{n}\cdot\mathbb E[g(S_n^x+s_n^x)]$ away from zero and $\sqrt{n}\cdot\mathbb E[h(S_n^x+s_n^x)]$
away from infinity.\\

\noindent {\bf Part 1.} Bound of $\sqrt{n}\cdot\mathbb E[g(S_n^x+s_n^x)]$ away from zero.\\

By the Fourier inverse theorem,
\begin{eqnarray*}
\sqrt{n}\cdot\mathbb E[g(S_n^x+s_n^x)]&=&\sqrt{n}\cdot\mathbb E\left[\int_{\R}\hat{g}(t)\exp(it(S_n^x+s_n^x))dt\right]\\
&=&\sqrt{n}\int_{\R}\hat{g}(t)\mathbb E[\exp(it(S_n^x+s_n^x))]dt\\
&=&\sqrt{n}\int_{\R}\hat{g}(t)\varphi_{S_n^x+s_n^x}(t)dt\\
&=&\sqrt{n}\int_{-\delta}^{\delta}\hat{g}(t)\varphi_{S_n^x+s_n^x}(t)dt+
   \sqrt{n}\int_{\delta<|t|<\Delta}\hat{g}(t)\varphi_{S_n^x+s_n^x}(t)dt.
\end{eqnarray*}
We claim that there are $\lambda<1$ and $n_0\ge 1$ such that
\begin{align}\label{characteristic far from zero}
|\varphi_{S_n^x+s_n^x}(t)|\le\lambda^n\,,\ \ \forall\,x\in X,\forall\,n>n_0,\forall\,\delta<|t|<\Delta.
\end{align}
To prove this, take $\varepsilon,\rho<1$ such that $[\varepsilon,2\varepsilon]\subset\phi(X)$ and
\begin{align*}
|\cos(st)|<\rho\,,\ \ \forall\,s\in[\varepsilon,2\varepsilon],\forall\,\delta<|t|<\Delta.
\end{align*}
Because $\phi$ is continuous and $T$ is uniquely ergodic, there is $n_0>0$ such that
\begin{align*}
\dfrac{\#\{0\le i<n:T^ix\in\phi^{-1}[\varepsilon,2\varepsilon]\}}{n}>\alpha\, ,\ \
\forall\,x\in X,\forall\, n>n_0,
\end{align*}
where $2\alpha=\nu(\phi^{-1}[\varepsilon,2\varepsilon])>0$. Thus, for every $x\in X,n>n_0$ and $\delta<|t|<\Delta$

\begin{eqnarray*}
|\varphi_{S_n^x+s_n^x}(t)|&=&|\varphi_{S_n^x}(t)|\\
                    &=&|\cos(c_0^xt)\cdots \cos(c_{n-1}^xt)|\\
                    &\le& \prod_{0\le i<n\atop{c_i^x\in[\varepsilon,2\varepsilon]}}|\cos(c_i^xt)|\\
                    &<& \rho^{\#\{0\le i<n:c_i^x\in[\varepsilon,2\varepsilon]\}}\\
                    &=&\rho^{\#\{0\le i<n:T^ix\in\phi^{-1}[\varepsilon,2\varepsilon]\}}\\
                    &<&\rho^{\alpha n}\\
                    &=&\lambda^n,
\end{eqnarray*}
where $\lambda=\rho^\alpha<1$. This establishes (\ref{characteristic far from zero}).
Then
\begin{align}\label{estimate integral close to zero}
\left|\sqrt{n}\int_{\delta<|t|<\Delta}\hat{g}(t)\varphi_{S_n^x+s_n^x}(t)dt\right|
<2\Delta\|\hat{g}\|_\infty\cdot\sqrt{n}\cdot\lambda^n.
\end{align}

To estimate the integral close to zero, first apply a change of variables:
\begin{eqnarray*}
\sqrt{n}\int_{-\delta}^{\delta}\hat{g}(t)\varphi_{S_n^x+s_n^x}(t)dt&=&
\int_{-\delta\sqrt{n}}^{\delta\sqrt{n}}\hat{g}\left(\dfrac{t}{\sqrt{n}}\right)\varphi_{S_n^x+s_n^x}\left(\dfrac{t}{\sqrt{n}}\right)dt\\
&=&\int_{-\delta\sqrt{n}}^{\delta\sqrt{n}}\hat{g}\left(\dfrac{t}{\sqrt{n}}\right)\varphi_{S_n^x}\left(\dfrac{t}{\sqrt{n}}\right)
\exp\left(it\dfrac{s_n^x}{\sqrt{n}}\right)dt\\
&=&\int_{-\delta\sqrt{n}}^{\delta\sqrt{n}}\hat{g}\left(\dfrac{t}{\sqrt{n}}\right)\varphi_{S_n^x}\left(\dfrac{t}{\sqrt{n}}\right)
\cos\left(t\dfrac{s_n^x}{\sqrt{n}}\right)dt.
\end{eqnarray*}
Let $\beta>0$ such that $\cos|_{[-\beta,\beta]}>1/2$, let
\begin{align*}
m_n^x=\min\left\{\dfrac{\beta\sqrt{n}}{s_n^x}\ ,\,\dfrac{\delta\sqrt{n}}{2}\right\},
\end{align*}
and divide the former integral into two parts accordingly to $m_n^x$:
\begin{eqnarray*}
\sqrt{n}\int_{-\delta}^{\delta}\hat{g}(t)\varphi_{S_n^x+s_n^x}(t)dt
&=&\int_{|t|<m_n^x}
\hat{g}\left(\dfrac{t}{\sqrt{n}}\right)\varphi_{S_n^x}\left(\dfrac{t}{\sqrt{n}}\right)
\cos\left(t\dfrac{s_n^x}{\sqrt{n}}\right)dt+\\
&&\int_{m_n^x<|t|<\delta\sqrt{n}}
\hat{g}\left(\dfrac{t}{\sqrt{n}}\right)\varphi_{S_n^x}\left(\dfrac{t}{\sqrt{n}}\right)\cos\left(t\dfrac{s_n^x}{\sqrt{n}}\right)dt\\
&&\\
&=&I_1+I_2.
\end{eqnarray*}
By the choice of $g$ and $\beta$, the fact that $m_n^x\rightarrow\infty$ as
$n\rightarrow\infty$, and Lemma \ref{lemma estimate characteristic}, we have
\begin{align*}
I_1\ge\dfrac{\hat g(0)}{4}\int_{|t|<m_n^x}\exp(-at^2)dt\ge\dfrac{\hat g(0)}{4}\int_{-1}^{1}\exp(-at^2)dt
\end{align*}
for every sufficiently large $n$ and arbitrary $x$. Similarly,
\begin{eqnarray*}
I_2&\ge&-\int_{m_n^x<|t|<\delta\sqrt{n}}\left|\hat{g}\left(\dfrac{t}{\sqrt{n}}\right)
        \varphi_{S_n^x}\left(\dfrac{t}{\sqrt{n}}\right)\cos\left(t\dfrac{s_n^x}{\sqrt{n}}\right)\right|dt\\
   &\ge&-\int_{m_n^x<|t|<\delta\sqrt{n}}\left\|\hat{g}\right\|_{\infty}\exp(-bt^2)dt\\
   &\ge&-\left\|\hat{g}\right\|_{\infty}\int_{|t|>m_n^x}\exp(-bt^2)dt.
\end{eqnarray*}
Thus
\begin{align*}
\sqrt{n}\int_{-\delta}^{\delta}\hat{g}(t)\varphi_{S_n^x+s_n^x}(t)dt
\ge\dfrac{\hat g(0)}{4}\int_{-1}^{1}\exp(-at^2)dt-\left\|\hat{g}\right\|_{\infty}\int_{|t|>m_n^x}\exp(-bt^2)dt
\end{align*}
is bounded away from zero if $n$ is large, uniformly in $x$. This, together with
(\ref{estimate integral close to zero}), proves Part 1.\\

\noindent {\bf Part 2.} Bound of $\sqrt{n}\cdot\mathbb E[h(S_n^x+s_n^x)]$ away from infinity.\\

Analogously as in Part 1, inequality (\ref{characteristic far from zero}) gives
\begin{align*}
\left|\sqrt{n}\int_{\delta<|t|<\Delta}\hat{h}(t)\varphi_{S_n^x+s_n^x}(t)dt\right|<2\Delta\|\hat{h}\|_\infty\cdot\sqrt{n}\cdot\lambda^n
\end{align*}
and Lemma \ref{lemma estimate characteristic} gives
\begin{eqnarray*}
\sqrt{n}\int_{-\delta}^{\delta}\hat{h}(t)\varphi_{S_n^x+s_n^x}(t)dt&=&
\int_{-\delta\sqrt{n}}^{\delta\sqrt{n}}\hat{h}\left(\dfrac{t}{\sqrt{n}}\right)\varphi_{S_n^x+s_n^x}\left(\dfrac{t}{\sqrt{n}}\right)dt\\
&\le&\int_{-\delta\sqrt{n}}^{\delta\sqrt{n}}\hat{h}\left(\dfrac{t}{\sqrt{n}}\right)\varphi_{S_n^x}\left(\dfrac{t}{\sqrt{n}}\right)dt\\
&\le&\|\hat{h}\|_{\infty}\int_{-\delta\sqrt{n}}^{\delta\sqrt{n}}\exp(-bt^2)dt\\
&\le&\|\hat{h}\|_{\infty}\int_{\R}\exp(-bt^2)dt,
\end{eqnarray*}
which is finite.
\end{proof}

The above proof is robust: given a compact set $\Lambda$, there are constants $K,n_0>0$ such that
Theorem \ref{main thm 4} is valid for any sequence $\{s_n^x+t\}_{n\ge 1}$, $t\in\Lambda$.

\section{Rational ergodicity: proof of Theorem \ref{main thm 2}}\label{section proof of thm 2}

Let $R_n$ be the return function of $\Omega\times X\times\left[-\frac{1}{2},\frac{1}{2}\right]$
with respect to $F$:
\begin{eqnarray*}
R_n(\omega,x,t)&=&
\#\left\{1\le i\le n:F^i(\omega,x,t)\in\Omega\times X\times\left[-\frac{1}{2},\frac{1}{2}\right]\right\}\\
               &=&\sum_{i=1}^n\chi_{[-1,1]}(S_i^x(\omega)+s_i^x+2t).
\end{eqnarray*}
We will show that
\begin{align}\label{condition rational ergodicity}
\int_{\Omega\times X\times\left[-\frac{1}{2},\frac{1}{2}\right]}R_n^2\lesssim n\lesssim\left(\int_{\Omega\times X\times\left[-\frac{1}{2},\frac{1}{2}\right]}R_n\right)^2.
\end{align}

Fix $x\in X$ and $t\in[-\frac{1}{2},\frac{1}{2}]$. By Theorem \ref{main thm 4}, we have
\begin{eqnarray*}
\int_{\Omega}R_n(\omega,x,t)&=&\sum_{i=1}^n\int_{\Omega}\chi_{[-1,1]}(S_i^x(\omega)+s_i^x+2t)\\
                            &&\\
                            &=&\sum_{i=1}^n\mathbb P\left[S_i^x\in[-1,1]-(s_i^x+2t)\right]\\
                            &\sim&\sum_{i=n_0+1}^n\mathbb P\left[S_i^x\in[-1,1]-(s_i^x+2t)\right]\\
                            &\sim&\sum_{n_0<i\le n}i^{-1/2}\\
                            &\sim&\int_1^n x^{-1/2}dx\\
                            &\sim&\sqrt{n}
\end{eqnarray*}
and thus
\begin{align}\label{estimate L^1-norm}
\int_{\Omega\times X\times\left[-\frac{1}{2},\frac{1}{2}\right]}R_n\gtrsim\sqrt{n}\,.
\end{align}

Now
\begin{eqnarray*}
\int_{\Omega}R_n(\omega,x,t)^2&=&\sum_{i=1}^n\int_{\Omega}\chi_{[-1,1]}(S_i^x(\omega)+s_i^x+2t)\\
& &+2\sum_{i<j}
\int_{\Omega}\chi_{[-1,1]}(S_i^x(\omega)+s_i^x+2t)\cdot\chi_{[-1,1]}(S_j^x(\omega)+s_j^x+2t)\\
&=&\int_{\Omega}R_n(\omega,x,t)\\
&&+2\sum_{i<j}\mathbb P[S_i^x\in[-1,1]-(s_i^x+2t),S_j^x\in[-1,1]-(s_j^x+2t)]\\
&\sim&\sqrt{n}+\sum_{i<j}\mathbb P[S_i^x\in[-1,1]-(s_i^x+2t),S_j^x\in[-1,1]-(s_j^x+2t)].
\end{eqnarray*}
Observe that $S_j^x(\omega)=S_i^x(\omega)+S_{j-i}^{T^ix}(\sigma^i\omega)$.
Because $S_i^x(\omega)$ depends on the coordinates $\omega_0,\ldots,\omega_{i-1}$
and $S_{j-i}^{T^ix}(\sigma^i\omega)$ depends on the coordinates
$\omega_{i},\ldots,\omega_{j-1}$, the random variables $S_i^x(\omega)$ and $S_{j-i}^{T^ix}(\sigma^i\omega)$
are independent. Thus, whenever $i>n_0$ and $j-i>n_0$
\begin{eqnarray*}
\mathbb P\left[
\begin{array}{c}
S_i^x\in[-1,1]-(s_i^x+2t),\\
S_j^x\in[-1,1]-(s_j^x+2t)
\end{array}
\right]&\le &
\mathbb P\left[
\begin{array}{c}
S_i^x\in[-1,1]-(s_i^x+2t),\\
S_{j-i}^{T^ix}\in[-2,2]-(s_i^x+s_j^x)
\end{array}
\right]\\
&&\\
&=&\mathbb P[S_i^x\in[-1,1]-(s_i^x+2t)]\times\\
&&\mathbb P[S_{j-i}^{T^ix}\in[-2,2]-(s_i^x+s_j^x)]\\
&&\\
     &\lesssim&i^{-1/2}\cdot (j-i)^{-1/2}.
\end{eqnarray*}
It follows that
\begin{eqnarray*}
\sum_{i<j}
\mathbb P\left[
\begin{array}{c}
S_i^x\in[-1,1]-(s_i^x+2t),\\
S_j^x\in[-1,1]-(s_j^x+2t)
\end{array}
\right]
&\lesssim&\sum_{1\le i<j\le n\atop{i,j-i>n_0}}i^{-1/2}\cdot (j-i)^{-1/2}\\
&\le& \left(\sum_{i=1}^n i^{-1/2}\right)^2\\
&&\\
&\sim& n.
\end{eqnarray*}
This implies
\begin{align*}
\int_{\Omega}R_n(\omega,x,t)^2\lesssim n
\end{align*}
for every $x\in X$ and every $t\in[-\frac{1}{2},\frac{1}{2}]$. Thus
\begin{align*}
\int_{\Omega\times X\times\left[-\frac{1}{2},\frac{1}{2}\right]}R_n^2\lesssim n
\end{align*}
which, together with (\ref{estimate L^1-norm}), establishes (\ref{condition rational ergodicity}).
This concludes the proof of Theorem \ref{main thm 2}.

\begin{remark}
A consequence of Theorem \ref{main thm 2} is that $F$ has generalized laws of large
numbers (see \S3.3 of \cite{aaronson1997introduction}). A natural step in the
program is to prove that it satisfies a second-order ergodic theorem in the sense
of \cite{fisher1992integer}.
\end{remark}

\section{Acknowledgements}

P.C. is supported by Fapesp-Brazil.
Y.L. is supported by the European Research Council, grant 239885.
E.P. is supported by CNPq-Brazil.

\bibliography{skew_products_in_infinite_measure}

\end{document}